\documentclass{article}
\usepackage{amssymb,amsmath}
\usepackage{graphicx}

\def\qed{\hfill {\hbox{${\vcenter{\vbox{               %HOLLOW SQUARE
   \hrule height 0.4pt\hbox{\vrule width 0.4pt height 6pt
   \kern5pt\vrule width 0.4pt}\hrule height 0.4pt}}}$}}}

\def\tr{\triangleright}
\def\bar{\overline}

\newtheorem{theorem}{Theorem}
\newtheorem{definition}{Definition}
\newtheorem{lemma}[theorem]{Lemma}

\newtheorem{corollary}[theorem]{Corollary}
\newtheorem{example}{Example}

\newtheorem{remark}{Remark}

\newenvironment{proof}[1][Proof]{\smallskip\noindent{\bf #1.}\quad}%
{\qed\par\medskip}

\date{}

\title{\Large \textbf{Augmented Biracks and their Homology}}

\author{
Jose Ceniceros
 \footnote{Email: \texttt{jcenic1@lsu.edu}  } \\Louisiana State University
\and
Mohamed Elhamdadi
 \footnote{Email: \texttt{emohamed@math.usf.edu}} \\University of South Florida
\and
Matthew Green
 \footnote{Email: \texttt{mgreen@mail.usf.edu}}\\ University of South Florida
\and 
Sam Nelson  \footnote{Email: \texttt{knots@esotericka.org}}\\Claremont McKenna College}

\begin{document}
\maketitle

\begin{abstract}
We introduce augmented biracks and define a (co)homology theory associated
to augmented biracks. The new homology theory extends the previously studied
Yang-Baxter homology with a combinatorial formulation for the boundary map and
specializes to $N$-reduced rack homology when the birack is a rack. We 
introduce augmented birack 2-cocycle invariants of classical and virtual knots
and links and provide examples.
\end{abstract}

\textsc{Keywords:} Biracks, rack homology, enhancements of counting 
invariants, cocycle invariants

\textsc{2010 MSC:} 57M27, 57M25

\section{Introduction}

\textit{Quandles}, an algebraic structure associated to oriented knots and 
links, were introduced by Joyce and Matveev independently in 1982 in 
\cite{J, Matv}. 
\textit{Racks}, the analogous structure associated to framed knots and links,
were introduced by Fenn and Rourke in 1992 in \cite{FR}. Soon thereafter,
\textit{biracks} were introduced in \cite{FRS1} and later the special cases 
known as \textit{biquandles} were studied in papers such as \cite{KR,FJK,NV}.

A homology and cohomology theory associated to racks was introduced in 
\cite{FRS}. In \cite{CJKLS}, a subcomplex of the rack chain complex was 
identified in the case when our rack is a quandle, and cocycles in the 
quotient complex (known as \textit{quandle cocycles}) were used to enhance the
quandle counting invariant, yielding \textit{CJKLS quandle cocycle invariants}.
In \cite{EN}, the degenerate subcomplex was generalized to the case of 
non-quandle racks with finite rack rank, yielding an analogous enhancement of
the rack counting invariant via \textit{$N$-reduced cocycles}. 

In \cite{CES1}, Yang-Baxter (co)homology was defined as a natural generalization
of the quandle (co)homology for biquandles, but the boundary map was difficult
to define combinatorially for arbitrary dimensions, making it impossible to 
define the degenerate subcomplex in general. Nevertheless, in the special case 
of biquandles, \textit{reduced 2--cocyles} were defined which allowed 
enhancement of the biquandle counting invariant, generalizing the CJKLS cocycle
invariants.

In this paper, we give a reformulation of the birack structure in terms of
actions of a set by an \textit{augmentation group} generalizing the augmented
quandle and augmented rack structures defined in \cite{J} and \cite{FR}. 
Our reformulation also restores the original approach taken in \cite{FRS1}
of using the ``sideways operations'' as the primary operations, as opposed
to the more usual approach of using the ``direct operations'' as primary.
This approach enables us to define a (co)homology theory for biracks with
a fully combinatorial formula for the boundary map, which we are able to
employ to identify the degenerate subcomplex associated to $N$-phone cord moves
for arbitrary biracks of finite
characteristic, generalizing the previous cases of quandle, $N$-reduced rack,
and Yang-Baxter (co)homology. As an application we use $N$-reduced augmented 
birack 2-cocycles to define cocycle enhancements of the birack counting 
invariant.

The paper is organized as follows. In Section \ref{AB} we introduce augmented
biracks. In Section \ref{ABH} we define augmented birack (co)homology and
discuss relationships with previously studied (co)homology theories. 
Section \ref{ABCI} deals with augmented birack cocycles and enhancement of the counting invariant.
In 
Section \ref{E} we give some examples of the new cocycle invariants and their
computation. We end in Section \ref{Q} with some questions for future research.

\section{Augmented Biracks}\label{AB}

\begin{definition}\label{def1}\textup{
Let $X$ be a set and $G$ be a subgroup of the group of
bijections $g:X\to X$. An \textit{augmented birack structure} on $(X,G)$
consists of maps $\alpha, \beta,\bar{\alpha},\bar{\beta}:X\to G$ 
(i.e., for each $x\in X$ we have bijections $\alpha_x:X\to X$, $\beta_x:X\to X$,
$\bar{\alpha_x}:X\to X$ and $\bar{\beta_x}:X\to X$) and a distinguished 
element $\pi\in G$ satisfying
\begin{itemize}
\item[(i)] For all $x\in X$, we have 
\[\alpha_{\pi(x)}(x)=\beta_x\pi(x) \quad \mathrm{and}\quad 
\bar{\beta_{\pi(x)}}(x)=\bar{\alpha_x}\pi(x), \]
\item[(ii)] For all $x,y\in X$ we have 
%\[\bar{\alpha_{\alpha_x(y)}}\beta_y(x)=x, \quad
%\bar{\beta_{\beta_x(y)}}\alpha_y(x)=x, \quad
%\alpha_{\bar{\alpha_x}(y)}\bar{\beta_y}(x)=x, \quad \mathrm{and} \quad
%\beta_{\bar{\beta_x}(y)}\bar{\alpha_y}(x)=x, 
%\]
\[\bar{\alpha_{\beta_x(y)}}\alpha_y(x)=x, \quad
\bar{\beta_{\alpha_x(y)}}\beta_y(x)=x,\quad
\alpha_{\bar{\beta_x(y)}}\bar{\alpha_y}(x)=x,
 \quad\mathrm{and}\quad 
\beta_{\bar{\alpha_x(y)}}\bar{\beta_y}(x)=x,
\]
and
%\item[(iii)] For all $x,y,z\in X$, we have
%\[\alpha_{\beta_y(x)}\alpha_y=\alpha_{\alpha_x(y)}\alpha_x,\quad
%\beta_{\alpha_x(z)}\alpha_x=\alpha_{\beta_z(x)}\beta_z,\quad \mathrm{and}\quad
%\beta_{\alpha_y(z)}\beta_y=\beta_{\beta_z(y)}\beta_z.
\item[(iii)] For all $x,y\in X$, we have
\[
\alpha_{\alpha_x(y)}\alpha_x=\alpha_{\beta_y(x)}\alpha_y,\quad
\beta_{\alpha_x(y)}\alpha_x=\alpha_{\beta_y(x)}\beta_y,\quad \mathrm{and}\quad
\beta_{\alpha_x(y)}\beta_x=\beta_{\beta_y(x)}\beta_y.
\]
\end{itemize}
}
\end{definition}

\begin{remark}\textup{
Alternatively, in definition \ref{def1} we could let $G$ be an arbitrary
group with an action $\cdot:G\times X\to X$ and maps
$\alpha,\beta,\bar{\alpha},\bar{\beta}:X\to G$ 
satisfying the listed conditions where $g(y)$ means $g\cdot y$.
}\end{remark}

%\begin{remark}\textup{
%Axiom (ii) is equivalent to the condition that the map $S:X\to X$ defined
%by $S(x,y)=(\alpha_x(y),\beta_y(x))$ is invertible with inverse map
%$S^{-1}(x,y)=(\bar{\beta_x}(y),\bar{\alpha_y}(x))$. 
%}\end{remark}

\begin{example}\textup{
Let $\tilde{\Lambda}=\mathbb{Z}[t^{\pm 1}, s, r^{\pm 1}]/(s^2-(1-t^{-1}r)s)$, 
let $X$ be a $\tilde{\Lambda}$-module and let $G$ be the group of invertible 
linear transformations of $X$. Then $(G,X)$ is an augmented biquandle with
%\[\begin{array}{rcl}
%\alpha_x(y) & = & ry \\
%\beta_y(x) & = & tx-tsy \\
%\bar{\alpha_x}(y) & = & sr^{-1}x+t^{-1}y\\
%\bar{\beta_y}(x) & = & r^{-1}x \\
%\pi(x) & = & (t^{-1}r+s)x. \\
%\end{array}\]
\[
\alpha_x(y) = ry, \quad
\beta_y(x) = tx-tsy, \quad
\bar{\alpha_y}(x) = r^{-1}x, \quad 
\bar{\beta_x}(y) = sr^{-1}x+t^{-1}y,
\quad\mathrm{and}\quad
\pi(x) = (t^{-1}r+s)x. 
\]
For example, we have
\[\beta_{\alpha_x(z)}\alpha_x(y) = \beta_{rz}(ry)= try-tsrz=r(ty-tsz)
=\alpha_{\beta_z(x)}\beta_z(y).\]
An augmented birack of this type is known as a \textit{$(t,s,r)$-birack}.
}\end{example}

\begin{example}
\textup{We can define an augmented birack structure symbolically on
the finite set $X=\{1,2,3,\dots,n\}$ by explicitly listing the maps
$\alpha_x,\beta_x:X\to X$ for each $x\in X$. This is conveniently done by 
giving a $2n\times n$ matrix whose upper block has $(i,j)$ entry $\alpha_j(i)$
and whose lower block has $(i,j)$ entry $\beta_j(i)$, which we might denote
by $M_{(G,X)}=\left[\begin{array}{c} 
\alpha_j(i) \\ \hline
\beta_j(i)
\end{array}\right]$. Such a matrix defines
an augmented birack with $G$ being the symmetric group $S_n$ provided the 
maps thus defined satisfy the augmented birack axioms; note that if the axioms
are satisfied, then the maps $\pi,\bar{\alpha_x}$ and $\bar{\beta_x}$ are
determined by the maps $\alpha_x,\beta_x$. For example, the matrix
\[M_{(G,X)}=\left[\begin{array}{ccc}
2 & 2 & 2 \\
1 & 1 & 1 \\
3 & 3 & 3 \\ \hline
2 & 3 & 1 \\
3 & 1 & 2 \\
1 & 2 & 3 \\
\end{array}\right]\]
encodes the $(t,s,r)$-birack structure on $X=\{1,2,3\}=\mathbb{Z}_3$ with
$t=1,s=2,r=2$.} 
\end{example}

An augmented birack defines a birack map 
$B:X\times X\to X\times X$ as defined in previous work by setting
\[B(x,y)=(\beta^{-1}_x(y),\alpha_{\beta^{-1}_x(y)}(x)).\]
The $G$-actions $\alpha_x,\beta_y$ are the components of the sideways map
in the notation of previous papers.

The geometric motivation for augmented biracks come from labeling semiarcs
in an oriented framed link diagram with elements of $X$; each Reidemeister 
move yields a set of necessary and sufficient conditions for labelings 
before and after the move to correspond bijectively.
\[\includegraphics{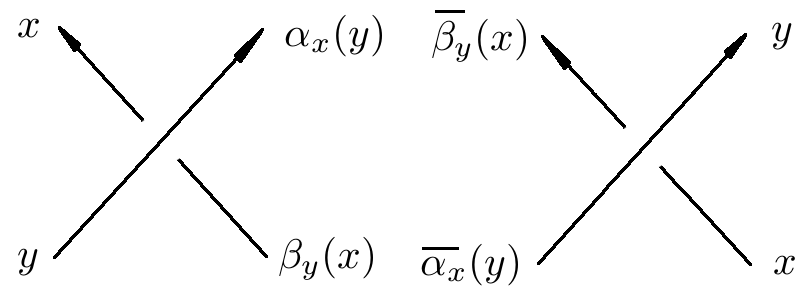}\]
The names are chosen so that if we orient a crossing, positive or negative,
with the strands oriented upward, then the unbarred actions go left-to-right
and the barred actions go right-to-left, with $\alpha$ and $\beta$ standing for
``above'' and ``below''\footnote{Thanks to Scott Carter for this observation!}.
Thus, $\alpha_x(y)$ is the result of $y$ going above $x$ left-to-right and
$\bar{\beta_y}(x)$ is the result of $x$ going below $y$ from right-to-left.

The element $\pi\in G$ is the \textit{kink map} which encodes the change of 
semiarc labels when going through a positive kink.
\[\includegraphics{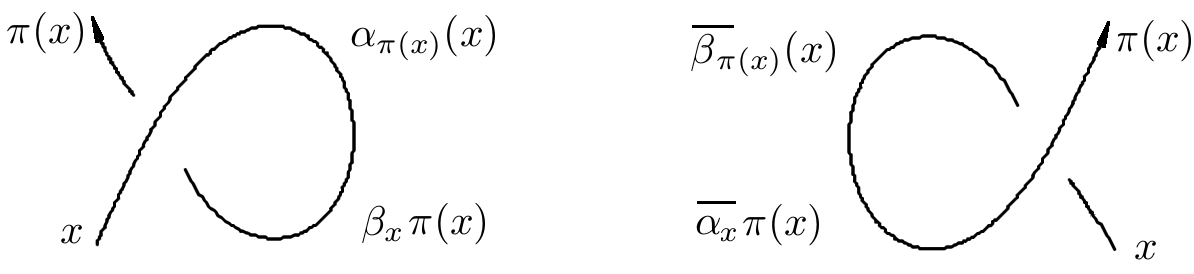}\]
In particular, each $(G,X)$-labeling of a framed oriented knot or link diagram
before a framed type I move corresponds to a unique $(G,X)$-labeling after
the move. If $\pi=1$ is the identity element in $G$, our augmented birack is an 
\textit{augmented biquandle}; labelings of an oriented link by an augmented
biquandle are independent of framing.

Axiom (ii) is equivalent to the condition that the map 
$S:X\times X\to X\times X$ defined by
\[S(x,y)=(\alpha_x(y),\beta_y(x))\]
is a bijection with inverse 
\[S^{-1}(y,x)=(\bar{\beta_y}(x),\bar{\alpha_x}(y)).\]
Note that the condition that the components of $S$ are bijective is not 
sufficient to make $S$ bijective; for instance, if $X$ is any abelian group, 
the map $S(x,y)=(\alpha_x(y),\beta_y(x))=(x+y,x+y)$ has bijective component maps
but is not bijective as a map of pairs.
The maps $\bar{\alpha_x},\bar{\beta_x}$ are the components of the inverse
of the sideways map; we can interpret them as labeling rules going right 
to left. At negatively oriented crossings, the top and bottom labels are 
switched.
\[\includegraphics{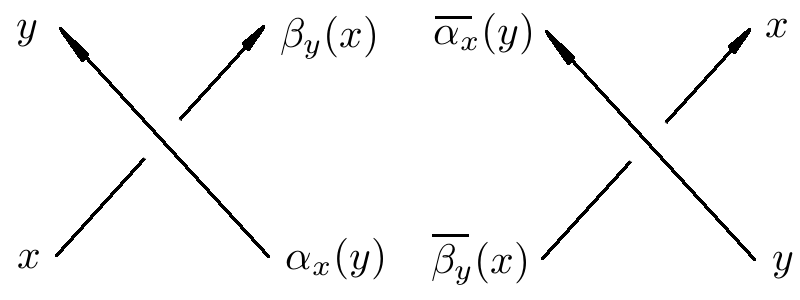}\]
Note that $(G,X)$-labelings of a framed knot or link correpond bijectively
before and after both forms of type II moves: \textit{direct type II} moves
where both strands are oriented in the same direction 
\[\includegraphics{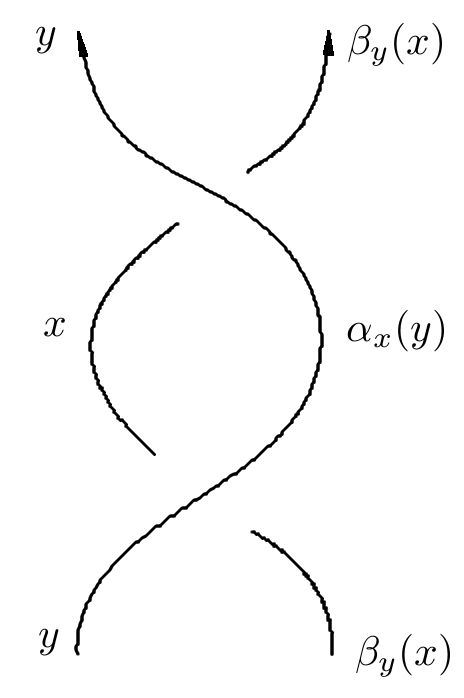} \quad\quad 
\raisebox{1in}{$\sim$}\quad\quad \includegraphics{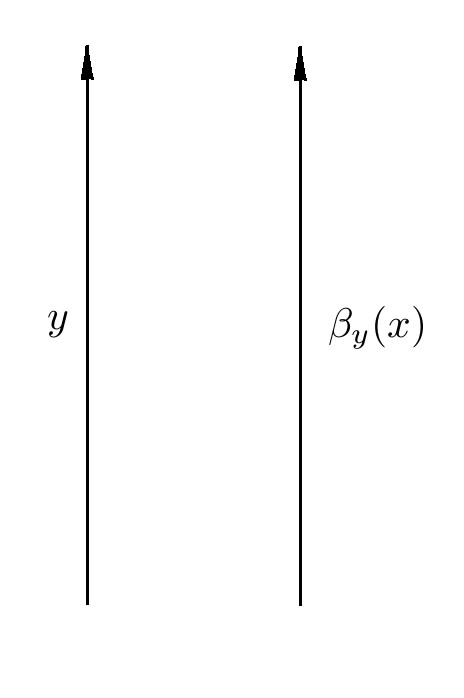}\]
and \textit{reverse type 
II} moves in which the strands are oriented in opposite directions.
\[\includegraphics{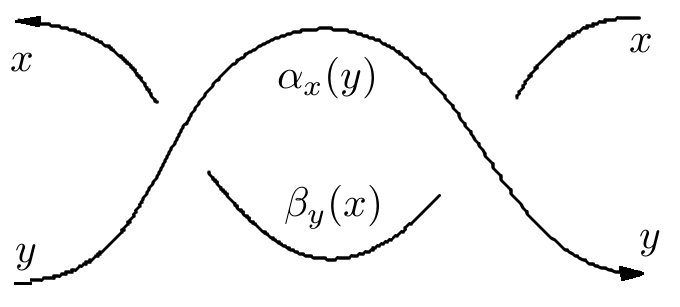} \quad
\raisebox{0.4in}{$\sim$}\quad \includegraphics{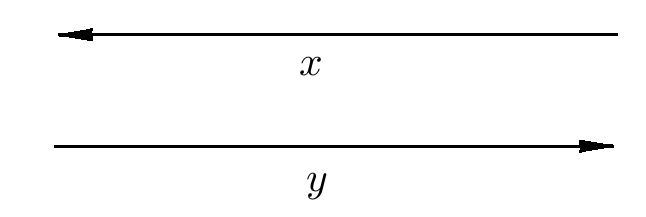}\]

Axiom (iii) encodes the conditions arising from the Reidemeister III move:
\[\includegraphics{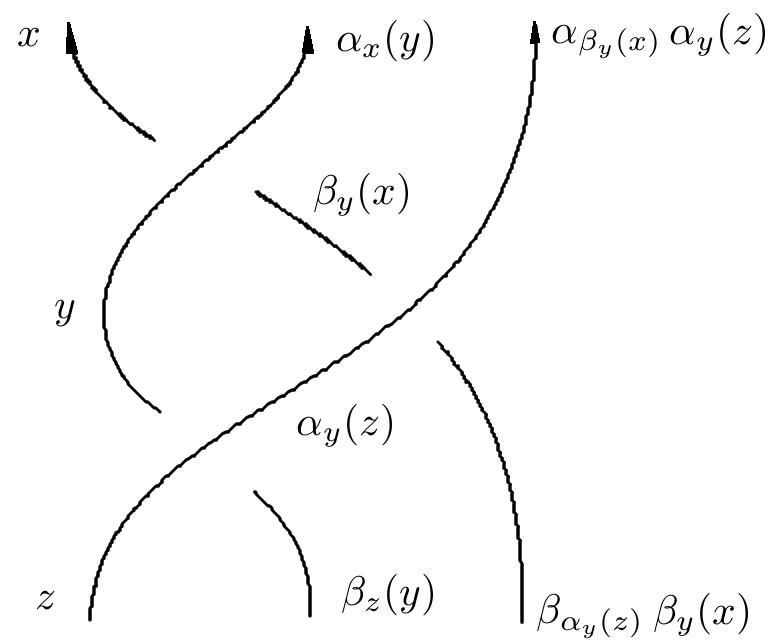} \quad\quad 
\raisebox{1in}{$\sim$}\quad\quad \includegraphics{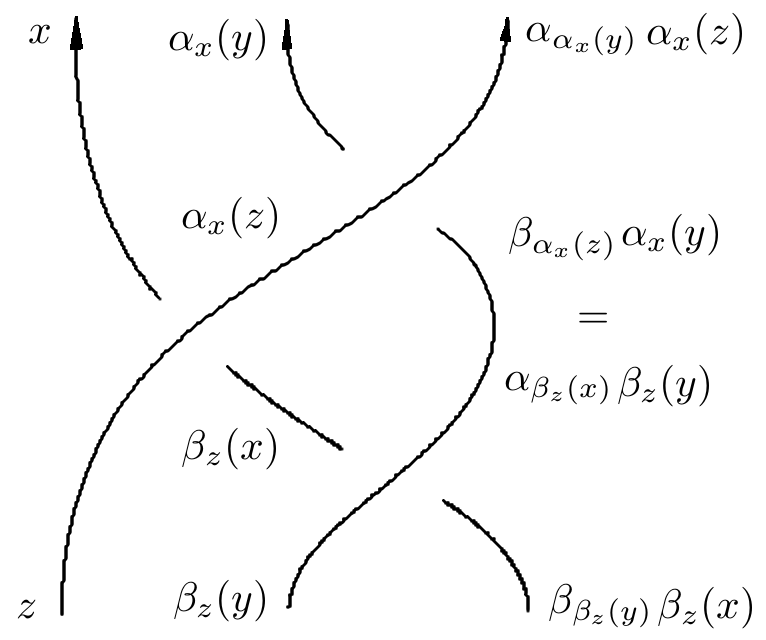}\]

Thus by construction we have
\begin{theorem}
If $L$ and $L'$ are oriented framed links related by oriented framed 
Reidemeister moves and $(G,X)$ is an augmented birack,
then there is a bijection between the set of labelings of $L$ by $(G,X)$,
denoted $\mathcal{L}(L,(G,X))$,
and the set of labelings and the set of labelings of $L'$ by $(G,X)$,
denoted $\mathcal{L}(L',(G,X))$.
\end{theorem}

\begin{remark}\textup{
Augmented biracks include several previously studies algebraic structures
as special cases.
\begin{itemize}
\item As mentioned above, an augmented birack is a birack with birack map
\[B(a,b)=(a^b,b_a)=(\beta_{a}^{-1}(b),\alpha_{\beta_{a}^{-1}(b)}(a))\] and is a 
\textit{biquandle} if $\pi=\mathrm{Id}:X\to X$,
\item An augmented birack in which $\alpha_x=\mathrm{Id}:X\to X$ for all $X$
is an \textit{augmented rack} with augmentation group $G$ and 
augmentation map $\delta(y)=\beta_y^{-1}$, as well as a \textit{rack}
with rack operation $x\tr y=\beta_y^{-1}(x)$,
\item An augmented birack in which $\alpha_x=\mathrm{Id}:X\to X$ for all $X$
and $\pi=\mathrm{Id}:X\to X$
is an \textit{augmented quandle} with augmentation group $G$ and 
augmentation map $\delta(y)=\beta_y^{-1}$, as well as a \textit{quandle}
with rack operation $x\tr y=\beta_y^{-1}(x)$.
\end{itemize}
}\end{remark}

Let us now consider the case when $X$ is a finite set.
For any framed oriented link $L$ with
$n$ crossings, there are at most $|X|^{2n}$ possible $(G,X)$-labelings of $L$,
so $\mathcal{L}(L,(G,X))$ is a positive integer-valued invariant of framed 
oriented links. More generally, if we choose an ordering on the components
of a $c$-component link $L$, then framings on $L$ correspond to elements
$\vec{w}\in\mathbb{Z}^c$ and we have a $c$-dimensional integral lattice
of framings of $L$.

If $X$ is a finite set, then $G$ is a subgroup of the symmetric group
$S_{|X|}$; in particular there is a unique smallest positive integer $N$
such that $\pi^N=1\in G$. This $N$ is called the \textit{characteristic} or
\textit{birack rank} of the augmented birack $(G,X)$. The value of 
$\mathcal{L}(L,(G,X))$ is unchanged by \textit{$N$-phone cord moves}:
\[\includegraphics{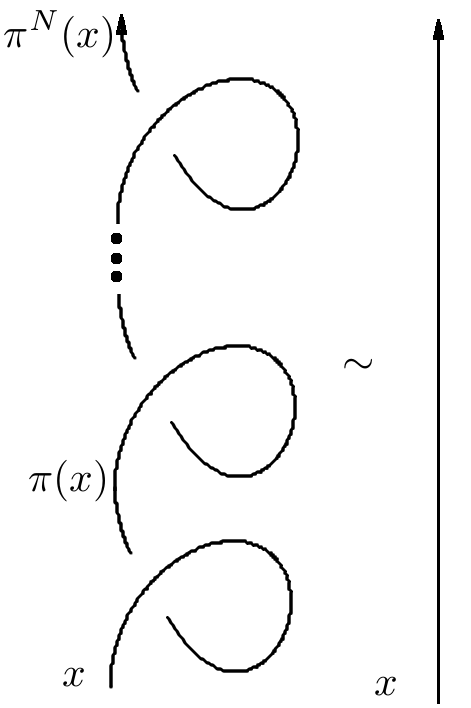}\]
In particular, framed oriented links which are equivalent by framed oriented
Reidemeister moves and $N$-phone cord moves have the same 
$\mathcal{L}(L,(G,X))$-values, and links which differ only by framing with
framing vectors equivalent mod $N$ have the same $\mathcal{L}(L,(G,X))$-values.
Hence, the $c$-dimensional lattice of values of $\mathcal{L}(L,(G,X))$
is tiled with a $c$-dimensional tile of side length $N$. We can thus obtain
an invariant of the unframed link by summing the $\mathcal{L}(L,(G,X))$-values
over a single tile.

\begin{definition}\textup{
Let $L$ be a link of $c$ components and $(G,X)$ a finite augmented birack.
Then the \textit{integral augmented birack counting invariant} of $L$ is the 
sum over one tile of framings mod $N$ of the numbers of $(G,X)$-labelings of
$L$. That is,
\[\Phi_{(G,X)}^{\mathbb{Z}}(L)
=\sum_{\vec{w}\in (\mathbb{Z}_N)^c} \mathcal{L}(L_{\vec{w}},(G,X)).\]
where $L_{\vec{w}}$ is $L$ with framing vector $\vec{w}$.
}\end{definition}

\section{Augmented Birack Homology}\label{ABH}

Let $(X,G)$ be an augmented birack. Let $C_n=\mathbb{Z}[X^n]$ be the free 
abelian group generated by ordered $n$-tuples of elements of $X$ and let 
$C^n(X)=\{f:C_n\to \mathbb{Z} \ | \ f \ \mathbb{Z}\mathrm{-linear\ transformation}\}$. For
$k=1,2,\dots, n$, define maps $\partial'_k,\partial''_k:C_n(X)\to C_{n-1}(X)$
by
\[\partial'_k(x_1.\dots,x_n)=(x_1,\dots,\widehat{x_k},\dots,x_n)\]
and
\[\partial''_k(x_1,\dots,x_n)=(\beta_{x_k}(x_1),\dots,\beta_{x_k}(x_{k-1}),\widehat{x_k},\alpha_{x_k}(x_{k+1}),\dots,\alpha_{x_k}(x_n))\]
where the\  $\widehat{\ }$\  indicates that the entry is deleted, i.e.
\[(x_1,\dots,\widehat{x_k},\dots,x_n)
=(x_1,\dots,x_{k-1},x_{k+1},\dots,x_n)
\]

\begin{theorem}\label{thm:boundary}
The map $\partial_n:C_n(X)\to C_{n-1}(X)$ given by
\[\partial_n(\vec{x})=\sum_{k=1}^{n} (-1)^k(\partial'_k(\vec{x})-\partial''_k(\vec{x}))\]
is a boundary map; the map $\delta^n:C^n(X)\to C^{n+1}(X)$ given by 
$\delta^n(f)=f\partial_{n+1}$ is the corresponding coboundary map. The quotient group 
$H_n(X)=\mathrm{Ker\ } \partial_n/\mathrm{Im\ } \partial_{n+1}$ is the 
\textit{$n$th augmented birack homology} of $(X,G)$, and the quotient group
$H^n(X)=\mathrm{Ker\ }\delta^n/\mathrm{Im\ } \delta^{n-1}$ is the 
\textit{$n$th augmented birack cohomology} of $(X,G)$.
\end{theorem}

To prove theorem \ref{thm:boundary}, we will find it convenient to first prove a few key lemmas.

\begin{lemma}
Let $j<k$. Then $\partial'_j\partial'_k(\vec{x})=\partial'_{k-1}\partial'_{j}$.
\end{lemma}

\begin{proof}
We compute
\begin{eqnarray*}
\partial'_j\partial'_k(\vec{x}) 
& = & \partial'_j(x_1,\dots,\widehat{x_k},\dots,x_n) \\
& = & \partial'_j(x_1,\dots,\widehat{x_j}\dots,\widehat{x_k},\dots,x_n) \\
\end{eqnarray*}
obtaining the input vector with the entries in the $j$th and $k$th positions 
deleted.
On the other hand, if we first delete the $j$th entry, each entry with
subscript greater than $j$ is now shifted into one lower position; in 
particular, $x_k$ is now in the $(k-1)$st position and we have
\begin{eqnarray*}
\partial'_{k-1}\partial'_{j}(\vec{x}) 
& = & \partial'_j(x_1,\dots,\widehat{x_j},\dots,x_n) \\
& = & \partial'_j(x_1,\dots,\widehat{x_j},\dots,\widehat{x_{k}},\dots,x_n) \\
\end{eqnarray*}
as required.
\end{proof}

\begin{corollary}\label{partial1}
The map $\partial':C_n\to C_{n-1}$ defined by $\displaystyle \sum_{k=1}^n (-1)^k\partial'(\vec{x})$ is a boundary map.
\end{corollary}

\begin{proof}
If we apply $\partial'$ twice, each term with first summation index less than
the second summation index is matched by an equal term with first summation
index greater than the second summation index but of opposite sign:
\begin{eqnarray*}
\partial'(\partial'(\vec{x}))
& = & \sum_{j<k}(-1)^{j+k}\partial'_j\partial'_k(\vec{x})+\sum_{j>k}(-1)^{j+k}\partial'_j(\partial'_k(\vec{x}))\\
& = & \sum_{j>k}(-1)^{j+k+1}\partial'_j\partial'_k(\vec{x})+\sum_{j>k}(-1)^{j+k}\partial'_j(\partial'_k(\vec{x}))\\
& = & 0.
\end{eqnarray*}
\end{proof}

\begin{lemma}
If $j<k$ we have $\partial'_{j}\partial''_k(\vec{x})=\partial''_{k-1}\partial'_{j}(\vec{x})$.
\end{lemma}

\begin{proof}
On the one hand, 
\begin{eqnarray*}
\partial'_{j}\partial''_k(\vec{x})
& = & \partial'_j(\beta_{x_k}(x_1),\dots,\beta_{x_k}(x_{k-1}),\widehat{x_k},\alpha_{x_k}(x_{k+1}),\dots,\alpha_{x_k}(x_n)) \\
& = & (\beta_{x_k}(x_1),\dots,\beta_{x_k}(x_{j-1}),\widehat{\beta_{x_k}(x_j)},\beta_{x_k}(x_{j+1}),\dots,\beta_{x_k}(x_{k-1}),\alpha_{x_k}(x_{k+1}),\dots,\alpha_{x_k}(x_n))) \\
\end{eqnarray*}
On the other hand, applying $\partial'_j$ first shifts $x_{k}$ into the $(k-1)$ 
position and we have
\begin{eqnarray*}
\partial_{k+1}''\partial'_j(\vec{x})
& = & \partial''_{k+1}(x_1,\dots,\widehat{x_j},\dots,x_n)\\
& = & (\beta_{x_k}(x_1),\dots,\beta_{x_k}(x_{j-1}),\widehat{x_j},\beta_{x_k}(x_{j+1}),\dots,\beta_{x_k}(x_{k-1}),\widehat{x_k},\alpha_{x_k}(x_{k+1}),\dots,\alpha_{x_k}(x_n))) \\
\end{eqnarray*}
as required.

\end{proof}

\begin{lemma}
If $j<k$ we have $\partial''_{j}\partial'_k(\vec{x})=\partial'_{k-1}\partial''_{j}(\vec{x})$.
\end{lemma}

\begin{proof}
On the one hand, 
\begin{eqnarray*}
\partial''_{j}\partial'_k(\vec{x})
& = & \partial''_j(x_1,\dots,\widehat{x_k},\dots,x_n) \\
& = & (\beta_{x_j}(x_1),\dots,\beta_{x_j}(x_{j-1}),\widehat{x_j},\alpha_{x_j}(x_{j+1}),\dots,\alpha_{x_j}(x_{k-1}),\widehat{x_k},\alpha_{x_j}(x_{k+1}),\dots,\alpha_{x_j}(x_n))) \\
\end{eqnarray*}
As above, applying $\partial''_j$ shifts the positions of the entries with
indices greater than $j$, and we have
\begin{eqnarray*}
\partial'_{k-1}\partial''_{j}(\vec{x}) & = & 
\partial'_{k-1}(\beta_{x_j}(x_1),\dots,\beta_{x_j}(x_{j-1}),\widehat{x_j},\alpha_{x_j}(x_{j+1}),\dots, \alpha_{x_j}(x_n))\\
& = & 
(\beta_{x_j}(x_1),\dots,\beta_{x_j}(x_{j-1}),\widehat{x_j},\alpha_{x_j}(x_{j+1}),\dots,\alpha_{x_j}(x_{k-1}),\widehat{\alpha_{x_j}(x_k)},\alpha_{x_j}(x_{k+1}),\dots \alpha_{x_j}(x_n))\\
\end{eqnarray*}
as required.
\end{proof}

The final lemma depends on the augmented birack axioms.

\begin{lemma}
If $j<k$ we have $\partial''_{j}\partial''_k(\vec{x})=\partial''_{k-1}\partial''_{j}(\vec{x})$.
\end{lemma}

\begin{proof}
We have
\begin{eqnarray*}
\partial''_{j}\partial''_{k}(\vec{x}) & = & 
\partial''_{j}(\beta_{x_k}(x_1),\dots,\beta_{x_k}(x_{k-1}),\widehat{x_k},\alpha_{x_k}(x_{k+1}),\dots,\alpha_{x_k}(x_n))\\
& = & (\beta_{\beta_{x_k}(x_j)}\beta_{x_k}(x_1),\dots,\beta_{\beta_{x_k}(x_j)}\beta_{x_k}(x_{j-1}),\widehat{\beta_{x_k}(x_j)},\alpha_{\beta_{x_k}(x_j)}\beta_{x_k}(x_{j+1}),\dots, \\
& & \quad\quad\quad\quad\quad
\alpha_{\beta_{x_k}(x_j)}\beta_{x_k}(x_{k-1}),\widehat{x_k},\alpha_{\beta_{x_k}(x_j)}\alpha_{x_k}(x_{k+1}),\dots,\alpha_{\beta_{x_k}(x_j)}\alpha_{x_k}(x_n))\\
\end{eqnarray*}
while again applying $\partial''_j$ first shifts the positions of the entries
with indices greater than $j$, and we have
\begin{eqnarray*}
\partial''_{k-1}\partial''_{j}(\vec{x}) 
& = &
\partial''_{k-1}(\beta_{x_j}(x_1),\dots,\beta_{x_j}(x_{j-1}),\widehat{x_j},\alpha_{x_j}(x_{j+1}),\dots,\alpha_{x_j}(x_n)) \\
& = & (\beta_{\alpha_{x_{j}(x_k)}}\beta_{x_j}(x_1),\dots,\beta_{\alpha_{x_{j}(x_k)}}\beta_{x_j}(x_{j-1}),\widehat{x_j},\beta_{\alpha_{x_{j}(x_k)}}\alpha_{x_j}(x_{j+1}),\dots, \\
& & \quad\quad\quad
\beta_{\alpha_{x_{j}(x_k)}}\alpha_{x_j}(x_{k-1}),\widehat{\alpha_{x_j}(x_k)},\alpha_{\alpha_{x_k}(x_j)}\alpha_{x_j}(x_{k+1}),\dots,
\alpha_{\alpha_{x_k}(x_j)}\alpha_{x_j}(x_{n})) \\
\end{eqnarray*}
and the two are equal after application of the augmented birack axioms.
\end{proof}

\begin{corollary}\label{partial2}
The map $\partial'':C_n\to C_{n-1}$ defined by $\displaystyle \partial''(\vec{x})=\sum_{k=1}^n (-1)^k \partial''_k(\vec{x})$ is a boundary map.
\end{corollary}

\begin{proof}
As with $\partial'$, we observe that every term in 
$\partial''_{n-1}\partial''_n(\vec{x})$ with $j<k$ is matched by an 
equal term with $j>k$ but with opposite sign.
\end{proof}

\begin{remark}\textup{
Corollary \ref{partial2} shows that the conditions in augmented birack axiom 
(iii)
are precisely the conditions required to make $\partial''$ a boundary map. 
This provides a non-knot theoretic alternative motivation for the augmented
birack structure.
}\end{remark}

%\begin{remark}\textup{
%Corollaries \ref{partial1} and \ref{partial2} say that these maps have the 
%\textit{Przytycki property}, i.e. that any
%}\end{remark}

\begin{proof} (of theorem \ref{thm:boundary})
We must check that $\partial_{n-1}\partial_n=0$. Our lemmas show that each term 
in the sum with $j<k$ is matched by an equal term with opposite sign with $j>k$.
We have
\begin{eqnarray*}
\partial_{n-1}(\partial_n(x_1,\dots, x_n))
& = & \partial_{n-1}\left(\sum_{k=0}^{n}(-1)^k\left(\partial'_k(\vec{x})-\partial''_k(\vec{x})\right)\right) \\
& = & \sum_{j=0}^{n-1}\left(\sum_{k=0}^{n}(-1)^{k+j}\left(\partial'_j\partial'_k(\vec{x})-\partial''_j\partial'_k(\vec{x})-\partial'_j\partial''_k(\vec{x})+\partial''_j\partial''_k(\vec{x})\right)\right) \\
& = & \sum_{j<k}(-1)^{k+j}\left(\partial'_j\partial'_k(\vec{x})-\partial''_j\partial'_k(\vec{x})-\partial'_j\partial''_k(\vec{x})+\partial''_j\partial''_k(\vec{x})\right) \\
& & + \sum_{j>k}(-1)^{k+j}\left(\partial'_j\partial'_k(\vec{x})-\partial''_j\partial'_k(\vec{x})-\partial'_j\partial''_k(\vec{x})+\partial''_j\partial''_k(\vec{x})\right) \\
& = & \sum_{j<k}(-1)^{k+j}\left(\partial'_j\partial'_k(\vec{x})-\partial''_j\partial'_k(\vec{x})-\partial'_j\partial''_k(\vec{x})+\partial''_j\partial''_k(\vec{x})\right) \\
& & + \sum_{j<k}(-1)^{k+j-1}\left(\partial'_j\partial'_k(\vec{x})-\partial''_j\partial'_k(\vec{x})-\partial'_j\partial''_k(\vec{x})+\partial''_j\partial''_k(\vec{x})\right) \\
& = & 0.
\end{eqnarray*}
\end{proof}

\begin{definition}\textup{
Let $(G,X)$ be an agumented birack of characteristic $N$.
Say that an element $\vec{v}$ of $C_n(X)$ is \textit{$N$-degenerate} if
$\vec{v}$ is a linear combination of elements of the form
\[\sum_{k=1}^N(x_1,\dots,x_{j-1},\pi^{k}(x_j),\pi^{k-1}(x_j),x_{j+2},\dots,x_n).\]
Denote the set of $N$-degenerate $n$-chains and $n$-cochains as $C^D_n(X)$
and $C_D^n(X)$ and the homology and cohomology groups, $H^{D}_n$ and $H_D^n$. 
}\end{definition}

\begin{theorem}
The sets of $N$-degenerate chains form a subcomplex of $(C_n,\partial)$.
\end{theorem} 

\begin{proof}
We must show that $\vec{v}\in C^D_n(X)$ implies 
$\partial(\vec{v})\in C^D_{n-1}(X)$.  Using linearity it is enough to prove that
\[\partial\left(\sum_{k=1}^N(x_1,\dots,x_{j-1},\pi^{k}(x_j),\pi^{k-1}(x_j),x_{j+2},\dots,x_n)\right)\] 
is $N$-degenerate.  Let $\displaystyle \vec{u}=\sum_{k=1}^N(x_1,\dots,x_{j-1},\pi^{k}(x_j),\pi^{k-1}(x_j),x_{j+2},\dots,x_n)$, we have:

%\[
%\begin{array}{rcl}
\begin{eqnarray*}
\partial (  \vec{u}) & = & 
\partial \left[
\sum_{k=1}^N 
(x_1,\dots,x_{j-1},\pi^{k}(x_j),\pi^{k-1}(x_j),x_{j+2},\dots,x_n)\right] \\ %\hfill (1)\\
%\nonumber \\
& = &\sum_{k=1}^N  \partial
(x_1,\dots,x_{j-1},\pi^{k}(x_j),\pi^{k-1}(x_j),x_{j+2},\dots,x_n) %\nonumber 
\\\\& = &  
\sum_{k=1}^N \left\{ \sum_{i=1}^{j-1} (-1)^{i} [ (x_1,\dots,\widehat{x_{i}},\dots, \pi^k(x_j),  \pi^{k-1}(x_j), x_{j+2}, \dots, x_n) \right.
%\nonumber \\ 
\\ & & 
\left. - (\beta_{x_i}(x_1) ,\dots, \beta_{x_i}(x_{i-1}) ,\widehat{x_i}, \alpha_{x_i} (x_{i+1}),\dots,  \alpha_{x_i}(\pi^k(x_j)), 
\alpha_{x_i}(\pi^{k-1}(x_j)), \alpha_{x_i}(x_{j+2}), \dots, \alpha_{x_i}(x_n)) ] \right\} %\nonumber 
\\ \\& & +
\left\{\sum_{k=1}^N 
\{ 
(-1)^j [ (x_1,\dots, x_{j-1}, \pi^ {k-1}(x_j), x_{j+2}, \dots, x_n) 
\right.  %\nonumber 
\\ & &
 - \left. %\left.
(\beta_{\pi^k(x_j)} (x_1) ,\dots, \beta_{\pi^k(x_j)} (x_{j-1})  , \alpha_{\pi^k(x_j)}( \pi^{k-1}(x_j)), \alpha_{\pi^k(x_j)} (x_{i+2}), \dots, \alpha_{\pi^k(x_j)} (x_n))\right]  
%\right. \right.
%\nonumber 
\\ &  & +
(-1)^{j+1}\left[ (x_1,\dots, x_{j-1}, \pi^k(x_j), x_{j+2}, \dots, x_n)
\right. %\nonumber 
\\& &  - 
\left. \left.(\beta_{\pi^{k-1}(x_j)} (x_1) ,\dots, \beta_{\pi^{k-1}(x_j)}  (x_{j-1}),  \beta_{\pi^{k-1}(x_j)}  (x_{j}) ,  \alpha_{\pi^{k-1}(x_j)} (x_{j+2}), \dots,  \alpha_{\pi^{k-1}(x_j)} ( x_n)) \right]  \right\}
%\nonumber 
\\ \\& & +
\sum_{k=1}^N \left\{\sum_{i=j+2}^{n} (-1)^{i} 
\left[ (x_1,\dots, \pi^{k}(x_j), \pi^{k-1}(x_j), \dots, \widehat{x_i},\dots, x_n)  \right. \right.%x_{j-1}, x_{j+1},\dots, x_n)  \right. \right.
%\nonumber 
\\&  &  \left.\left.
 - ( \beta_{x_i}(x_1) ,\dots,  \beta_{x_i}(x_{j-1}) ,  \beta_{x_i}(\pi^k(x_j)) ,  \beta_{x_i}(\pi^{k-1}(x_j)) , \dots, \widehat{x_i},  
%  \right.\right. \\
% & & \left.\left. \hskip 3in
  \alpha_{x_i}(x_{i+1}), \dots, \alpha_{x_i}(x_n)) \right]
\right\} \hfill \quad\quad\quad(1)%\\
\end{eqnarray*}
%\end{array}\]
where as usual $(x_1,\dots, \widehat{x_i},\dots, x_n)$ means $(x_1,\dots, x_{i-1},x_{i+1},\dots,x_n)$.
%\end{eqnarray}
Now the rest of the proof is based on the following two facts: (1) $\pi^N=1$ and (2) $\alpha_{\pi^k(x)}(\pi^{k-1}(x))=\beta_{\pi^{k-1}(x)}(\pi^k(x))$ which is obtained by induction from axiom (i) in the definition \ref{def1} (of augmented birack).

 The following sum vanishes:
%\[
%\begin{array}{rcl}
\begin{eqnarray*}
& &
\sum_{k=1}^N 
\left\{
 \left[ (x_1,\dots, x_{j-1}, \pi^ {k-1}(x_j), x_{j+2}, \dots, x_n) 
\right.\right. %\nonumber 
\\ 
& &
%& &
 - \left. \left.
(\beta_{\pi^k(x_j)} (x_1) ,\dots, \beta_{\pi^k(x_j)} (x_{j-1})  , \alpha_{\pi^k(x_j)}( \pi^{k-1}(x_j)), \alpha_{\pi^k(x_j)} (x_{i+2}), \dots, \alpha_{\pi^k(x_j)} (x_n))\right]  \right.
%\nonumber 
\\ 
& &
%&  & 
- \left[ (x_1,\dots, x_{j-1}, \pi^k(x_j), x_{j+2}, \dots, x_n)
\right. %\nonumber 
\\ & & %& & 
 - 
\left. \left.(\beta_{\pi^{k-1}(x_j)} (x_1) ,\dots, \beta_{\pi^{k-1}(x_j)}  (x_{j-1}),  \beta_{\pi^{k-1}(x_j)}  (x_{j}) ,  \alpha_{\pi^{k-1}(x_j)} (x_{j+2}), \dots,  \alpha_{\pi^{k-1}(x_j)} ( x_n)) \right]  \right\}
\end{eqnarray*}
%\end{array}\]
because $\alpha_{\pi^k(x)}(\pi^{k-1}(x))=\beta_{\pi^{k-1}(x)}(\pi^k(x))$ and $\pi^N=1$.  The rest of the sums can be written as combination of degenerate elements as in the proof of theorem 2 in \cite{EN} by the authors.

%thus
%\[\sum_{k=1}^N \left[ (x_1,\dots, x_{i-1}, x_i^{\tr (k+1)}, x_{i+2}, \dots, x_n)- (x_1,\dots, x_{i-1}, x_i^{\tr k}, x_{i+2}, \dots, x_n)\right]=0
%\]
%and
%\[\begin{array}{l}
%(x_1 \tr x_i^{\tr k} ,\dots, x_{i-1} \tr x_i^{\tr k}, x_i^{\tr (k+1)}, x_{i+2}, \dots, x_n) \\
%=  (x_1 \tr x_i^{\tr (k+1)} ,\dots, x_{i-1} \tr x_i^{\tr (k+1)}, x_i^{\tr k}\tr x_i^{\tr (k+1)},x_{i+2}, \dots, x_n).
%\end{array}
%\]
%This last difference is zero since by [15, Lemma 1] we have 
%$\forall u \in X,\; u \tr x_i^{\tr (k+1)}= u \tr (x_i^{\tr k} \tr x_i^{\tr k})= u \tr x_i^{\tr k}$ and similarly 
%$x_i^{\tr k} \tr x_i^{\tr (k+1)}= x_i^{\tr k} \tr x_i^{\tr k}$.
%
%Using self-distributivity, the last term in equation (1) can be re-written as (so it makes the last sum fit 
%the definition of degenerate chains)
%\[
%(x_1 \tr x_j,\dots, x_{i-1} \tr x_j,x_i^{\tr k} \tr x_j,( x_i^{\tr k} \tr x_j)\tr (x_i^{\tr k} \tr x_j ), 
%\dots, x_{j-1} \tr x_j,x_{j+1}, \dots, x_n) 
%\]
\end{proof}

\begin{definition}\textup{
The quotient groups $H_n^{NR}(X)=H_n(X)/H^D_n(X)$ and $H^n_{NR}(X)=H^n(X)/H_{D}^n(X)$
are the \textit{$N$-Reduced Birack Homology} and 
\textit{$N$-Reduced Birack Cohomology} groups.
}\end{definition}

\section{Augmented Birack Cocycle Invariants}\label{ABCI}

In this section we will use augmented birack cocycles to enhance the 
augmented birack counting invariant analogously to previous work.

Let $L_{\vec{w}}$ be an oriented framed link diagram with framing vector $\vec{w}$
and a labeling $f\in\mathcal{L}(L_{\vec{w}},(G,X))$ by an augmented birack $(G,X)$ 
of characteristic $N$. For a choice of $\phi\in H^2_{NR}$, we define an integer-valued 
signature of the labeling called a \textit{Boltzmann weight} by adding contributions from
each crossing as pictured below. Orienting the crossing so that both strands are 
oriented upward, each crossing contributes $\phi$ evaluated on the pair of labels
on the left side of the crossing with the understrand label listed first.
\[\includegraphics{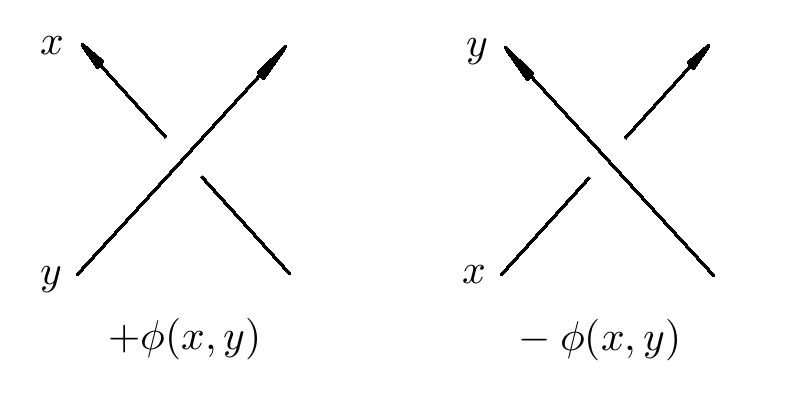}\]
Then as we can easily verify, the Boltzmann weight $BW(f)=\sum_{\mathrm{crossings}} \pm\phi(x,y)$ 
is unchanged by framed oriented
Reidemeister moves and $N$-phone cord moves. Starting with move III, note that 
$\phi\in H^2(x)$ implies that
\begin{eqnarray*}
(\delta^2\phi)(x,y,z) & = & \phi(\partial_2(x,y,z)) \\
& = & \phi((y,z)-(\alpha_x(y),\alpha_x(z))-(x,z)+(\beta_y(x),\alpha_y(z))+(x,y)-(\beta_z(x),\beta_z(y))) \\
& = & \phi(y,z)-\phi(\alpha_x(y),\alpha_x(z))-\phi(x,z)+\phi(\beta_y(x),\alpha_y(z))
+\phi(x,y)-\phi(\beta_z(x),\beta_z(y)) \\ 
& = & 0
\end{eqnarray*}
and in particular we have
\[\phi(y,z)+\phi(\beta_y(x),\alpha_y(z))+\phi(x,y)
 =  \phi(\alpha_x(y),\alpha_x(z))+\phi(x,z)+\phi(\beta_z(x),\beta_z(y)).\]
Then both sides of the Reidemeister III move contribute the same amount to the
Boltzmann weight:
\[\includegraphics{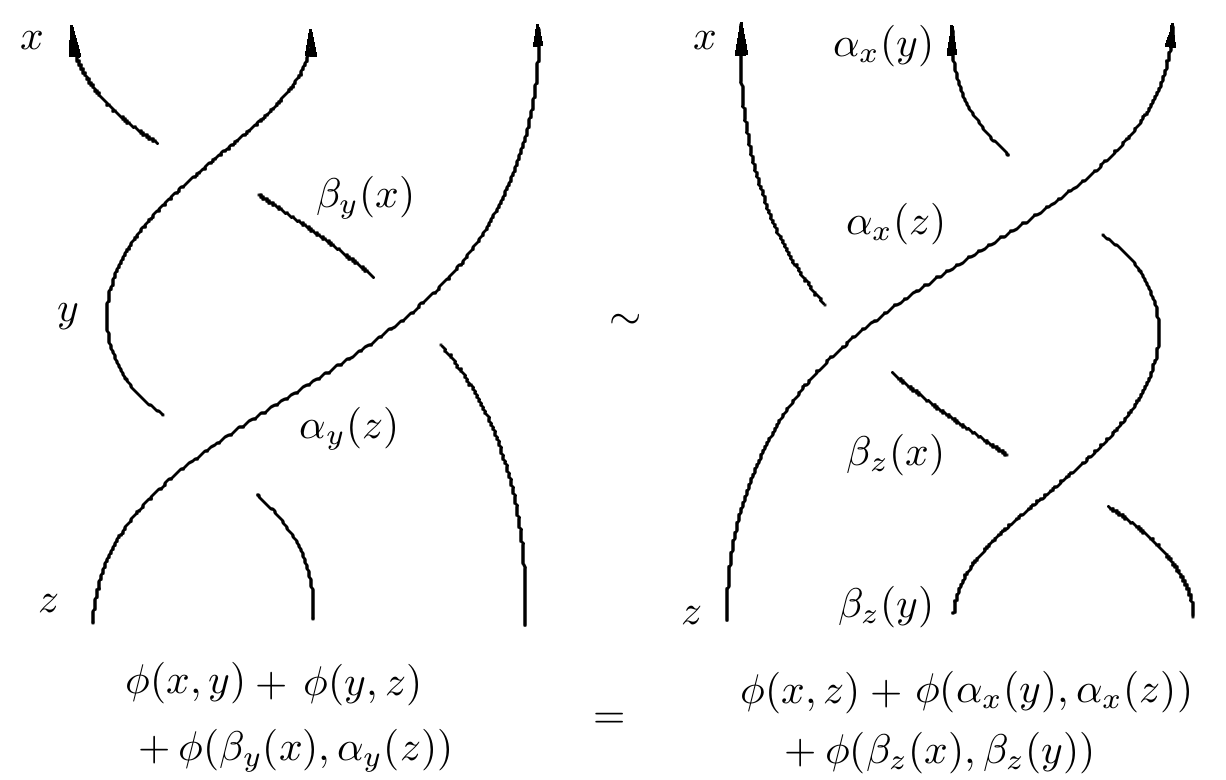}\]
Both sides of both type II moves contribute $0$ to the Bolztmann weight:
\[\includegraphics{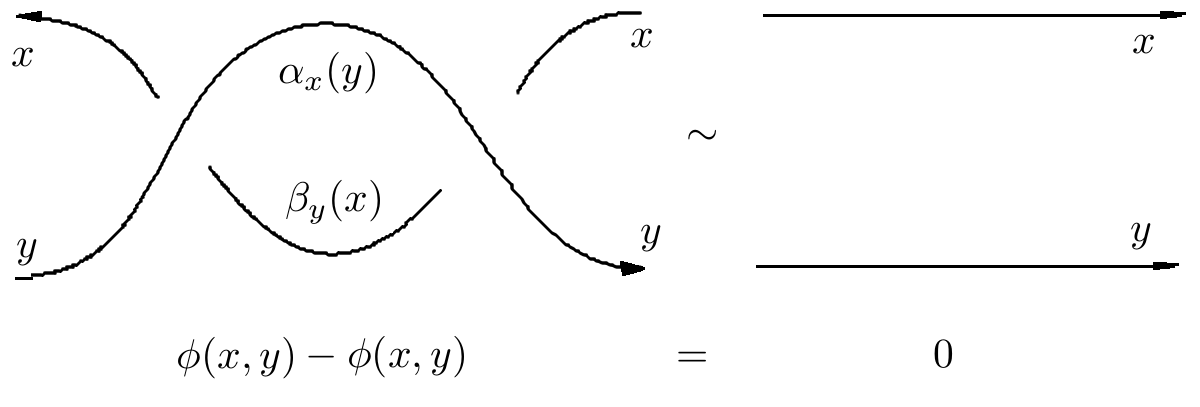}\]
\[\includegraphics{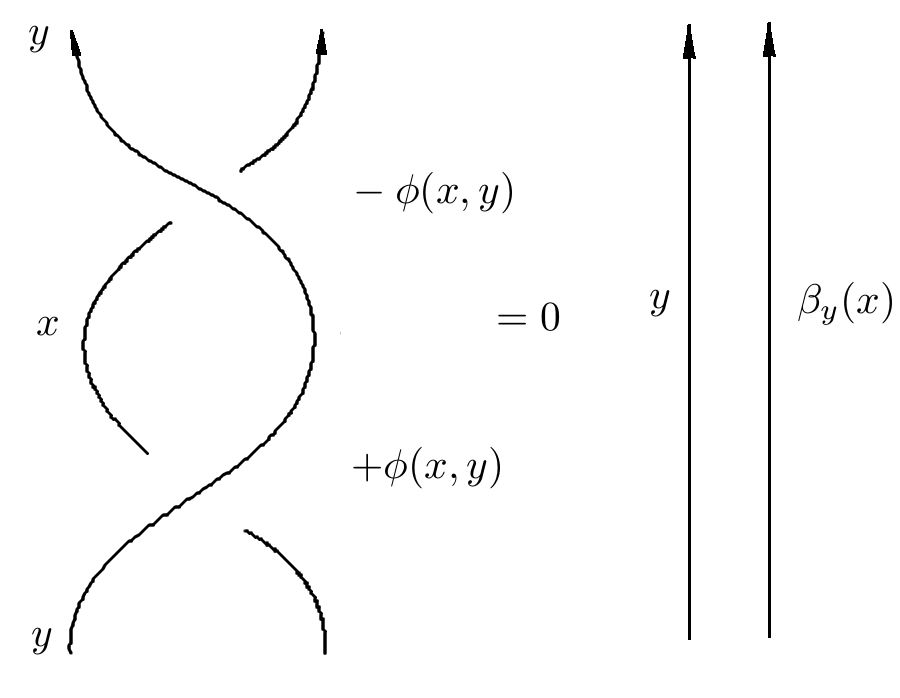}\]
Similarly, both sides of the framed type I moves contribute zero; here
we use the alternate form of the framed type I move for clarity, with 
$y=\bar{\alpha_x}\pi(x)=\bar{\beta_{\pi(x)}}(x)$:
\[\includegraphics{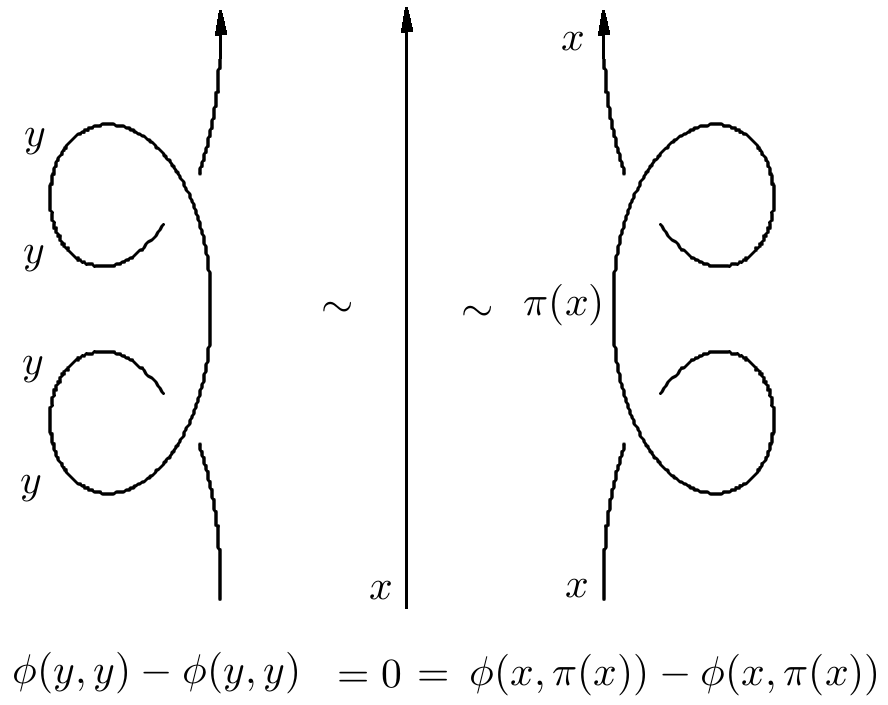}\] 
Finally, the $N$-phone cord move contributes a degenerate $N$-chain:
\[\includegraphics{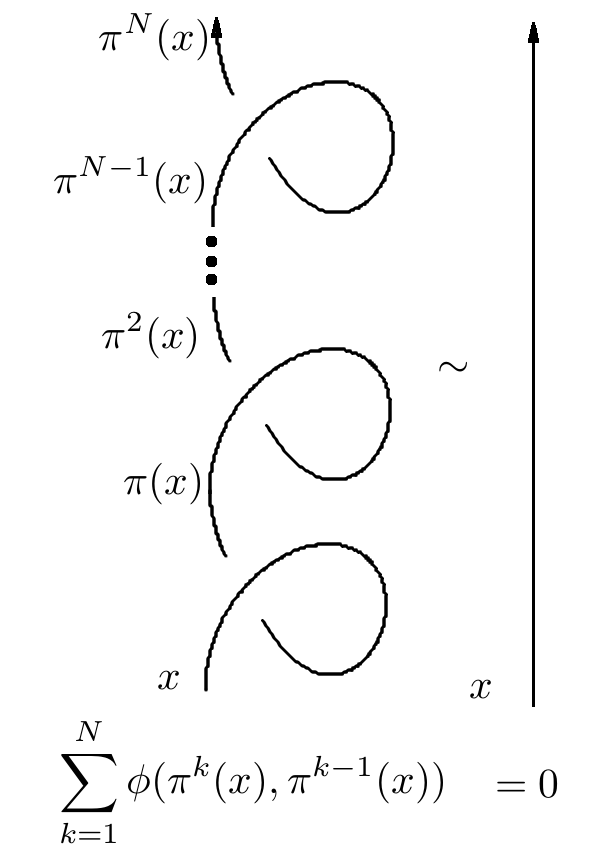}\] 

Putting it all together, we have our main result:
\begin{theorem}
Let $L$ be an oriented unframed link of $c$ components and $(G,X)$ a finite
augmented birack of characteristic $N$. For each $\phi\in H^2_{NR}(X)$, the multiset
$\Phi_{\phi}^M(L)$ and polynomial $\Phi_{\phi}(L)$ defined by
\[\Phi_{\phi}^M(L)=\{BW(f)\ |\ f\in \mathcal{L}(L_{\vec{w}},(G,X)), \vec{w}\in (\mathbb{Z}_N)^c\}\]
and
\[\Phi_{\phi}(L)=\sum_{\vec{w}\in (\mathbb{Z}_N)^c} \left(\sum_{ f\in \mathcal{L}(L_{\vec{w}},(G,X))} u^{BW(f)}\right)\]
are invariants of $L$ known as the \textit{augmented birack 2-cocycle 
invariants} of $L$.
\end{theorem}

\begin{remark}\textup{
We note that if $\phi\in H^2(X)$ then the corresponding quantities, 
\[\Phi_{\phi}^M(L_{\vec{w}})=\{BW(f)\ |\ f\in \mathcal{L}(L_{\vec{w}},(G,X))\} \quad
\mathrm{and} \quad \Phi_{\phi}(L_{\vec{w}})=\sum_{ f\in \mathcal{L}(L_{\vec{w}},(G,X))} u^{BW(f)},\] 
are invariants of $L_{\vec{w}}$ as a framed link.}
\end{remark}

\begin{remark}\textup{
If $L$ is a virtual link, $\Phi_{\phi}^M(L)$ and $\Phi_{\phi}$ are invariants
of $L$ under virtual isotopy via the usual convention of ignoring
the virtual crossings. 
}\end{remark}

As in quandle homology, we have 
\begin{theorem}
Let $(G,X)$ be an augmented birack.
If $\phi\in H^2(X)$ is a coboundary, then for any $(G,X)$-labeling $f$ of a
framed link $L_{\vec{w}}$ the Boltzmann weight $BW(f)=0$.
\end{theorem}

\begin{proof}
If $\phi\in H^2(X)$ is a coboundary, then there is a map $\psi\in H^1$ such that
$\psi=\delta^2\phi=(\phi\delta_2)$. Then for any $(x,y)$ we have
\[
\phi(x,y)  =  \psi(\delta_2(x,y)) 
 = \psi(y) -\psi(\alpha_{x}(y))-\psi(x)+\psi(\beta_y(x)) 
\]
and the Boltzmann weight can be pictured at a crossing as below.
\[\includegraphics{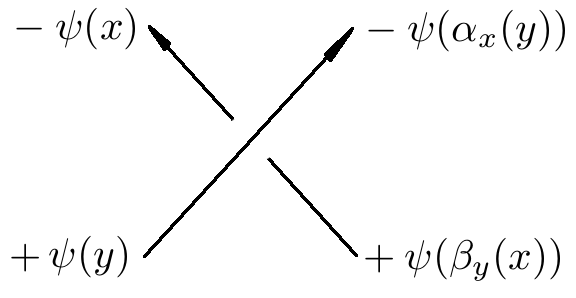}\]
In particular, every semiarc labeled $x$ contributes a $+\psi(x)$ at 
its tail and a $-\psi(x)$ at its head, so each semiarc contributes zero
to the Boltzmann weight.
\end{proof}

\begin{corollary}
Cohomologous cocycles define the same $\Phi_{\phi}(L)$ and $\Phi_{\phi}^M(L)$
invariants.
\end{corollary}

\section{Examples}\label{E}

In this section we collect a few examples of the augmented birack cocycle 
invariants and their computation.

\begin{example}\textup{
Let $X=\{1,2,3,4\}$ be the set of four elements and $G=S_4$ the group
of permutations of $X$. The pair $(G,X)$ has augmented birack structures
including 
\[M_{(G,X)}=\left[\begin{array}{cccc}
2 & 3 & 3 & 2 \\
4 & 1 & 1 & 4 \\
1 & 4 & 4 & 1 \\
3 & 2 & 2 & 3 \\ \hline
3 & 2 & 2 & 3 \\
1 & 4 & 4 & 1 \\
4 & 1 & 1 & 4 \\
2 & 3 & 3 & 2 \\
\end{array}\right].\]
This augmented birack has kink map $\pi=(14)(23)$ and hence characteristic 
$N=2$. Thus, to find a complete tile of labelings of a link $L$, we'll need to 
consider diagrams of $L$ with framing vectors 
$\vec{w}\in(\mathbb{Z}_2)^2=\{(0,0),(0,1),(1,0),(1,1)\}$. Then for instance 
the Hopf link $L=L2a1$ has no labelings in framings $(0,0)$, $(1,0)$ 
and $(0,1)$ and sixteen labelings in framing $(1,1)$, for a counting invariant 
value of $\Phi^{\mathbb{Z}}_{(G,X)}(L2a1)=16+0+0+0=16$.
\[\includegraphics{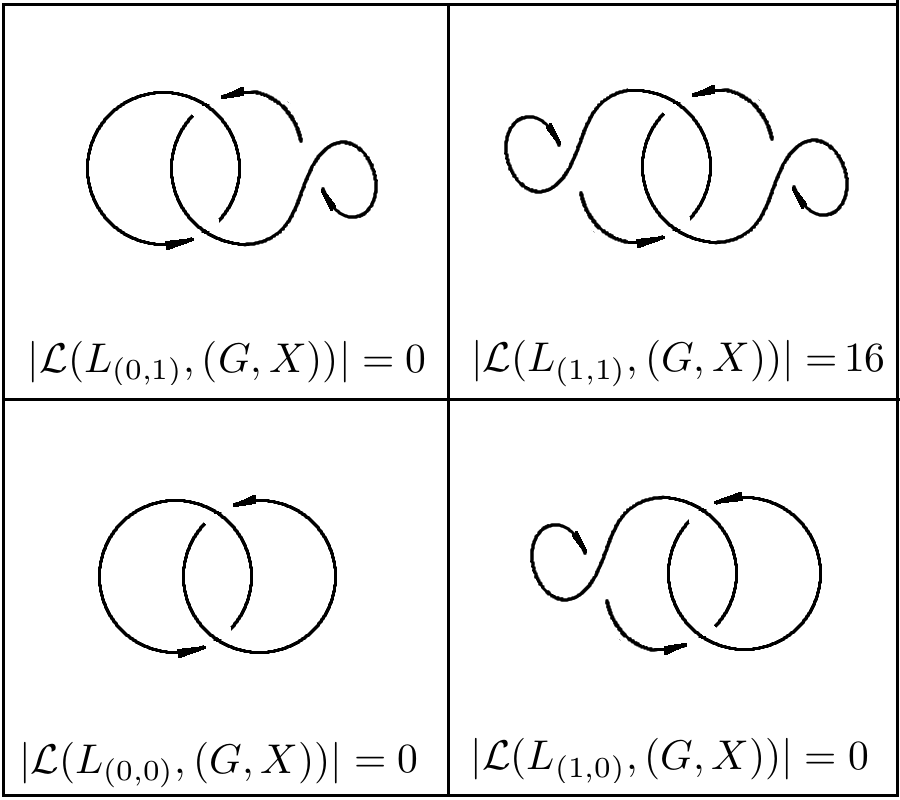}\]
Note that $C^2(X)$ has $\mathbb{Z}$--basis $\{\chi_{ij}\ |\ i,j=1,2,3,4,5\}$ where
\[\chi_{(i,j)}((i',j'))=\left\{\begin{array}{ll}
1 & (i,j)=(i',j') \\
0 & (i,j)\ne(i',j'). \\
\end{array}\right.\]
The function $\phi:X\times X\to \mathbb{Z}$ defined by 
\[\phi=\chi_{(2,1)}+\chi_{(2,4)}+\chi_{(3,1)}+\chi_{(3,4)}\]
is an $N$-reduced 2-cocycle in $H^2_{NR}(G,X)$. We then compute $\Phi_{\phi}(L)$
by finding the Boltzmann weight for each labeling. 
\[\includegraphics{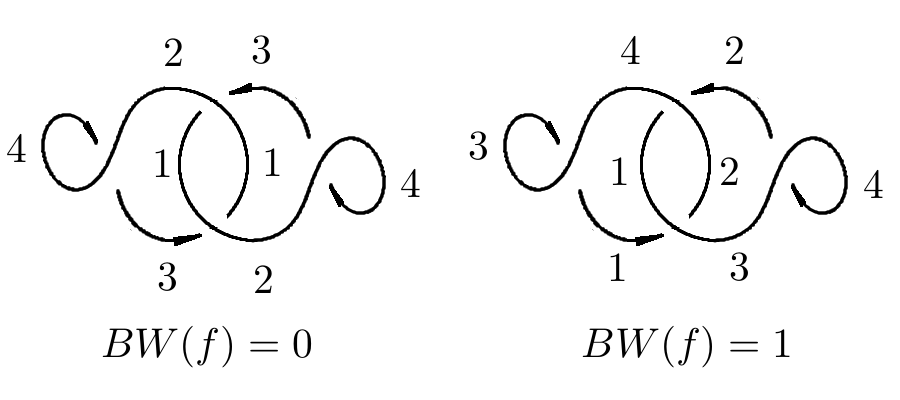}\]
In the labeling on the left, we have 
\[BW(f)=\phi(1,1)+\phi(1,1)+\phi(3,2)+\phi(3,2)=0+0+0+0=0\]
and in the labeling on the right we have
\[BW(f)=\phi(1,2)+\phi(2,1)+\phi(2,3)+\phi(1,4)=0+1+0+0=1.\]
Repeating for all 14 other labelings, we get $\Phi_{\phi}(L2a1)=8+8u$. 
Similarly the unlink $L0a1$ and $(4,2)$-torus link $L4a1$ have counting 
invariant value $\Phi^{\mathbb{Z}}_{(G,X)}(L0a1)=\Phi^{\mathbb{Z}}_{(G,X)}(L4a1)=
16$ with respect to $(G,X)$ but augmented birack cocycles invariant values
$\Phi_{\phi}(L0a1)=16$ and $\Phi_{\phi}(L4a1)=8+8u^2$ respectively. 
}\end{example}

\begin{example}\textup{
Now let $X$ be the augmented birack with matrix
\[\left[\begin{array}{ccccc}
4 & 1 & 1 & 4 & 4 \\
3 & 3 & 3 & 3 & 3 \\
2 & 2 & 2 & 2 & 2 \\
5 & 4 & 4 & 5 & 5 \\
1 & 5 & 5 & 1 & 1 \\ \hline
4 & 1 & 1 & 4 & 4 \\
2 & 3 & 3 & 2 & 2 \\
3 & 2 & 2 & 3 & 3 \\
5 & 4 & 4 & 5 & 5 \\
1 & 5 & 5 & 1 & 1 
\end{array}\right].\]
Augmented birack cocyle invariants are defined for virtual knots by the
usual convention of ignoring virtual crossings. Using our \texttt{Python}
code available at \texttt{http://www.esotericka.org}, we computed the values 
of $\Phi_{\phi}(K)$ for all virtual knots $K$ with up to four crossings as 
collected in the knot atlas \cite{KA} with the cocycle 
$\phi=\chi_{(1,4)}+\chi_{(1,5)}+\chi_{(5,4)}$; the results are collected in the 
table below.
\[
\begin{array}{r|l}
\Phi_{\phi}(K) & K \\ \hline
2+3u^{-2} & 4.1,4.3,4.7,4.25,4.37,4.43,4.48,4.53,4.73,4.81,4.82,4.87,4.89,4.100\\
2+3u^{-1} & 3.2,3.3,3.4,3.5,3.7,4.4,4.5,4.9,4.11,4.15,4.18,4.19,4.20,4.23,4.27,
4.29,4.30,4.33,4.34,4.35, \\
& 4.40,4.42,4.44,4.47,4.52,4.54,4.60,4.61,4.62,4.63,
4.65,4.69,4.74,4.78,4.79,4.80,4.83,4.85,\\ &  
4.86,4.91,4.93,4.94,4.96,4.97,4.102,4.106\\
5 & 3.1,3.6,4.2,4.6,4.8,4.10,4.12,4.13,4.14,4.16,4.17,4.22,4.24,4.28,4.31,4.32,
4.38,4.39,4.41, \\
& 4.45,4.46,4.49,4.50, 4.51,4.55,4.56,4.57, 4.58,4.59,4.64,4.66,4.67,4.68,4.70,4.71,4.72,4.75,\\
& 4.76,4.77,4.84,4.88,4.90,4.92,4.95,4.98,4.99,4.101, 4.103,4.104,4.105,4.107,4.108 \\
2+3u & 2.1,4.21,4.26,4.36\\
\end{array}\]
Our \texttt{Python} comptations, confirmed independently by Maple, show that for this augmented birack and cocycle,
$\Phi_{\phi}(K)=5$ for classical knots
$K$ with up to 8 crossings; it seems likely that that this is true for all 
classical knots $K$. For classical links $L$, however, $\Phi_{\phi}(L)$ is quite
nontrivial. Our values for $\Phi_{\phi}(L)$ for prime classical links with up to
$7$ crossings as listed in the knot atlas are below.
\[\begin{array}{r|l}
\Phi_{\phi}(L) & L \\ \hline
6u^{-2}+19 & L7n1 \\
12u^{-1}+41+30u+6u^2 & L7a7 \\
25 & L5a1,L7a1, L7a3, L7a4, L7n2 \\
125 & L6a4 \\
7+6u & L2a1, L7a5, L7a6 \\
19+6u^2 & L4a1, L6a1, L7a2 \\  
7+6u^3 & L6a2, L6a3 \\
29+36u+18u^2+6u^3 & L6a5, L6n1 \\ 
%47+21u^2+3u^4 & L6a4 
\end{array}\]
We also note that this example demonstrates that the invariant is sensitive to 
orientation, as for instance the Hopf link oriented to make both crossings 
positive has $\Phi_{\phi}$ value $7+6u$, while reversing one component yields
a  $\Phi_{\phi}$ value of $7+6u^{-1}$.
}\end{example}

\section{Questions}\label{Q}

In this section we collect a few open questions for future research.

In the case of quandle homology, many results are known involving the 
long exact sequence, the delayed Fibonacci sequence in the dimensions
of the homology groups for certain quandles, etc. Which of these results
extend to augmented birack homology?

In \cite{CN}, Yang-Baxter (co)homology was paired with $S$-(co)homology
to define an enhancement of the virtual biquandle counting invariant.
A future paper will consider the relationship between augmented birack
(co)homology and $S$-cohomology.

Can 3-cocycles in augmented birack homology be used to define invariants
of knotted surfaces in $\mathbb{R}^4$, analogously to the quandle case?

\noindent
\textsc{Department of Mathematics \\
Louisiana State University \\
Baton Rouge, LA 70803-4918}, \\\\
\textsc{Department of Mathematics, \\
University of South Florida, \\
4202 E Fowler Ave., \\
Tampa, FL 33620}\\\\
and\\\\
\textsc{Department of Mathematics, \\
Claremont McKenna College, \\
850 Columbia Ave., \\
Claremont, CA 91711}

\end{document}